\documentclass{amsart}
\usepackage{amsmath,amsthm,verbatim,amssymb,amsfonts,amscd,  graphicx}
\usepackage{xcolor}
\usepackage{graphics}

\usepackage{euscript, enumerate}
\usepackage{url,hyperref}

\newcommand{\eps}{\varepsilon}

\newcommand{\be}{\begin{equation}}
\newcommand{\ba}{\begin{aligned}}
\newcommand{\bee}{\begin{equation*}}
\newcommand{\ee}{\end{equation}}
\newcommand{\ea}{\end{aligned}}
\newcommand{\eee}{\end{equation*}}
\newcommand{\bea}{\begin{equation} \begin{aligned} }
\newcommand{\eea}{\end{aligned}\end{equation} }

\theoremstyle{plain}
\newtheorem{theorem}{Theorem}[section]

\newtheorem{corollary}[theorem]{Corollary}

\newtheorem{prop}[theorem]{Proposition}
\newtheorem{conjecture}[theorem]{Conjecture}

\theoremstyle{remark}

\theoremstyle{definition}

\numberwithin{equation}{section}

\begin{document}
\title{A note on the selfsimilarity of limit flows}
\author{Beomjun Choi, Robert Haslhofer, Or Hershkovits}

\begin{abstract}
It is a fundamental open problem for the mean curvature flow, and in fact for many partial differential equations, whether or not all blowup limits are selfsimilar. In this short note, we prove that for the mean curvature flow of mean convex surfaces all limit flows are selfsimilar (static, shrinking or translating) if and only if there are only finitely many spherical singularities. More generally, using the solution of the mean convex neighborhood conjecture for neck singularities \cite{CHHW}, we establish a local version of this equivalence for neck singularities in arbitrary dimension. In particular, we see that the ancient ovals occur as limit flows if and only if there is a sequence of spherical singularities converging to a neck singularity.
\end{abstract}
\maketitle

\section{Introduction}

Let $M\subset\mathbb{R}^3$ be a smooth closed embedded mean-convex surface. Then there exists a unique evolution $\mathcal{M}=\{M_t\}_{t\geq 0}$ by mean curvature flow, which can be viewed both as level set flow and as unit-regular integral Brakke flow (see e.g. \cite{HK_meanconvex,White_size} for background on mean-convex mean curvature flow).\\

A fundamental open problem for the mean curvature flow, and in fact for many partial differential equations, concerns the selfsimilarity of blowup limits. To describe this, for any $X=(x,t)\in\mathcal{M}$ and $\lambda>0$ we denote by
$\mathcal{M}_{X,\lambda}=\mathcal{D}_{\lambda}(\mathcal{M}-X)$
the flow which is obtained from $\mathcal{M}$ by translating $X$ to the origin and parabolically dilating by $\lambda$. Now, given $X\in\mathcal{M}$, and arbitrary sequences $X_i\to X$ and $\lambda_i\to\infty$, one can always pass to a subsequential limit of $\mathcal{M}_{X_i,\lambda_i}$. Any such limit $\mathcal{M}^\infty$ is called a limit flow at $X$.

\begin{conjecture}[{selfsimilarity of limit flows, see e.g. \cite[Conj. 6]{Ilmanen_problems} and \cite[Conj. 1 and Conj. 3]{White_nature}}]\label{conj_selfsim_limit} 
For the mean curvature flow of a smooth closed embedded mean-convex surface all limit flows are selfsimilar (either static, shrinking or translating).
\end{conjecture}

We note that while tangent flows, i.e. the special case of limit flows where $X_i=X$ is fixed, are always selfsimilarly shrinking by Huisken's monotonicity formula \cite{Huisken_monotonicity}, there is no general mechanism which indicates whether or not limit flows are selfsimilar. The only other available tool in this direction is Hamilton's Harnack inequality \cite{Hamilton_Harnack}, but to apply it to get a selfsimilarly translating limit flow the sequence $X_i$ has to be chosen in some very special way.\footnote{An illustrative example is  the rotationally symmetric degenerate neckpinch from Angenent-Velazquez \cite{AV}. To get the selfsimilarly translating bowl soliton as limit flow one has to choose the sequence $X_i$ along suitable ``tip points".} In particular, White \cite{White_nature} (see also Haslhofer-Hershkovits \cite{HH_ovals}) constructed examples of ancient convex noncollapsed mean curvature flows that are not selfsimilar; these examples are called the ancient ovals, and it is an open problem whether or not they can occur as limit flows.\\

It seems to be known among some experts in the field that there is a close relationship between the potential scenario of ancient ovals as a limit flow and the potential scenario of infinitely many singularities (see also \cite{AAG,CM_arrival,White_ICM} for related discussions regarding the finiteness of singular times, and a heuristic relationship with analytic functions). The purpose of this short note is to state and prove a precise equivalence. To this end, let us first recall another well-known conjecture concerning the finiteness of spherical singularities:

\begin{conjecture}[{finiteness of spherical singularities, see e.g. \cite[Conj.~Q]{Ilmanen_monograph} and \cite[Conj. on p.~1237]{Wang}}]\label{conj_finiteness}
For the mean curvature flow of a smooth closed embedded mean-convex surface there are only finitely many spherical singularities.
\end{conjecture}

Here, we say that the flow $\mathcal{M}$ has a spherical singularity at $X$, if the tangent flow at $X$ is a round shrinking sphere, c.f. \cite{HK_meanconvex,White_size,White_nature}. It is known that there can be at most countably many spherical singularities, c.f. \cite{White_stratification}. Conjecture \ref{conj_finiteness} is known to be true for smooth rotationally symmetric initial surfaces by work of Altschuler-Angenent-Giga \cite{AAG}. The smoothness assumption cannot be dropped. Indeed, Miura \cite{Miura} gave examples of singular mean convex initial surfaces, whose flows have countably many spherical singularities.\\

The following theorem shows that Conjecture \ref{conj_selfsim_limit} (selfsimilarity of limit flows) and Conjecture \ref{conj_finiteness} (finiteness of spherical singularities) are equivalent. Specifically, we prove:

\begin{theorem}\label{thm_equivalence} 
Let $\mathcal{M}=\{M_t\}_{t\geq 0}$ be the mean curvature flow of a smooth closed embedded mean-convex surface. Then all limit flows of $\mathcal{M}$ are selfsimilar if and only if $\mathcal{M}$ has only finitely many spherical singularities.
\end{theorem}

In particular, as a corollary (of the proof) we obtain:

\begin{corollary}\label{cor_ovals} 
The ancient ovals occur as a limit flow of $\mathcal{M}$ if and only if there is a sequence of spherical singularities converging to a cylindrical singularity.
\end{corollary}

More generally, we establish a local version of this equivalence that holds for neck singularities in arbitrary dimension. To describe this, given any smooth closed embedded hypersurface $M=\partial K\subset\mathbb{R}^{n+1}$, as in \cite{HW} we denote by $(\mathcal{M},\mathcal{K})$ and $(\mathcal{M}',\mathcal{K'})$ the outer and inner flow, respectively.\footnote{In particular, we have that $M_t=M_t'$ and $K_t'=\overline{\mathbb{R}^{n+1}\setminus K_t}$ for at least as long as the level set flow does not fatten.} As in \cite[Def. 1.17]{CHHW} we say that the mean curvature evolution of $M$ has an inwards or outwards neck singularity at $X$ if the rescaled flow $\mathcal{D}_\lambda(\mathcal{K}-X)$ or $\mathcal{D}_\lambda(\mathcal{K}'-X)$, respectively, converges for $\lambda\to \infty$ locally smoothly with multiplicity $1$ to a round shrinking solid cylinder $\{\bar{B}^n(\!\!\sqrt{2(n-1)|t|})\times \mathbb{R} \}_{t<0}$, up to rotation. Using these notions, we can now state our local result in arbitrary dimensions:

\begin{theorem}\label{thm_equivalence_general} 
Suppose the mean curvature evolution of a smooth closed embedded hypersurface $M$ has an inwards (respectively outwards) neck singularity at $X_0=(x_0,t_0)$. Then there exists an $\varepsilon>0$ such that $[t_0-\varepsilon,t_0+\varepsilon]\ni t\mapsto M_t\cap B_\varepsilon(x_0)$ (respectively $M_t'\cap B_\varepsilon(x_0)$) has only finitely many spherical singularities if and only if the ancient ovals do not occur as limit flow of $\mathcal{M}$ (respectively $\mathcal{M}'$) at $X_0$.
\end{theorem}

To establish these equivalences we proceed as follows. First, we show that if there is a sequence of spherical singularities converging to a cylindrical singularity, then we can rescale by suitable factors around suitable centers to get a limit flow that is not selfsimilar. Next, we show that the ovals cannot occur as limit flow if there are only finitely many spherical singularities. Together with recent results from Brendle-Choi \cite{BC}, Angenent-Daskalopoulos-Sesum \cite{ADS} and Choi-Haslhofer-Hershkovits-White \cite{CHHW} we can then conclude the proof. For the sake of exposition, we first give the proof in the simpler $2$-dimensional global setting of Theorem \ref{thm_equivalence}, and afterwards describe the necessary generalizations needed for the proof of Theorem \ref{thm_equivalence_general}.\\

Finally, it seems reasonable to expect that similar results should hold for $3$-dimensional Ricci flow through singularities, as introduced by Kleiner-Lott \cite{KL_singular}.\\

\bigskip

\noindent\textbf{Acknowledgments.} RH has been partially supported by an NSERC Discovery Grant (RGPIN-2016-04331) and a Sloan Research Fellowship. OH has been partially supported by a Koret Foundation early career scholar award.\\

\bigskip

\section{The proofs}

To prove Theorem \ref{thm_equivalence}, we proceed by establishing the following three propositions.

\begin{prop}\label{prop1}
If there is a sequence of spherical singularities converging to a cylindrical singularity, then there occurs a limit flow that is not selfsimilar.
\end{prop}

\begin{proof}
Fix $\eps>0$ small enough. Given $X\in\mathcal{M}$ we consider the rescaled flow $\mathcal{M}_{X,1/r_{\alpha}}$ on diadic annuli of radius $r_\alpha=2^\alpha$, where $\alpha\in\mathbb{Z}$.
Denote by $S(X)$ the last spherical scale, i.e. the supremum of $r_\alpha$ such that $\mathcal{M}$ is $\eps$-close around $X$ at scale $r_\alpha$ to a round shrinking sphere.
Denote by $Z(X)$ the first cylindrical scale, i.e. the infimum of $r_\alpha$ such that $\mathcal{M}$ is $\eps$-close around $X$ at scale $r_\alpha$ to a round shrinking cylinder.

Now assume $X_i\in \mathcal{M}$ is a sequence of spherical singularities converging to a cylindrical singularity $X\in \mathcal{M}$. Since we have spherical singularities it holds that $S(X_i)>0$. On the other hand, since the flow has a cylindrical singularity at $X$ and since $X_i\to X$, for $i$ large enough we have $Z(X_i)<\infty$, and in fact
\begin{equation}
\lim_{i\to \infty} Z(X_i)=0.
\end{equation}
By Huisken's monotonicity formula \cite{Huisken_monotonicity}, we clearly have $S(X_i)\leq Z(X_i)$. By quantitative differentiation, see e.g. \cite{CHN}, the ratio between these scales is bounded, i.e.
\begin{equation}
\frac{Z(X_i)}{S(X_i)}\leq C.
\end{equation}
Consider the rescaled flows $\mathcal{M}_{X_i,S(X_i)^{-1}}$ and pass to a subsequential limit $\mathcal{M}^\infty$. Then, $\mathcal{M}^\infty$ is $\eps$-spherical at scale $1$ and $\eps$-cylindrical at some scale between $1$ and $C$. Hence, the limit flow $\mathcal{M}^\infty$ is not selfsimilar.
\end{proof}

\begin{prop}\label{prop2}
If there are only finitely many spherical singularities, then ancient ovals do not occur as limit flow.
\end{prop}

\begin{proof}
Suppose towards a contradiction that there are $X_i=(x_i,t_i) \to X_0=(x_0,t_0)$ and $\lambda_i\to\infty$ such that $\mathcal{M}_{X_i,\lambda_i}$ converges to an ancient oval which becomes extinct at $t=0$ at the origin. Since $\lambda_i\cdot(M_{t_i-\lambda_i^{-2}}-x_i)$ converges to the time $t=-1$ slice of the oval, which is compact and strictly convex, there exists an $i_0$ such that, for $i>i_0$, the oval-like component of $\lambda_i\cdot (M_{t_i-\lambda_i^{-2}}-x_i)$ is strictly convex and hence the flow of that component has a spherical singularity at some $(y_i,s_i) \in \mathcal{M}_{X_i, \lambda_i}$. Consequently, the unrescaled flow $\mathcal{M}$ has a spherical singularities at $(x_i+\lambda_i^{-1} y_i, t_i+\lambda_i^{-2} s_i)=:(x_i',t_i')$. 

By the convergence to the oval, we have $y_i\to 0$ and $s_i \to 0$, and thus it is not hard to see that $\mathcal{M}_{(x'_i,t'_i),\lambda_i}$ also converges to an oval, which becomes extinct at $(x_0,t_0)$. Since there are only finitely many spherical singularities, and since $(x_i',t_i')$ converges to $(x_0,t_0)$, we infer that $(x_i',t_i')=(x_0,t_0)$ for large $i$.

In summary, $\mathcal{M}_{(x_0,t_0),\lambda_i}$ converges for $i\to\infty$ to an ancient oval; this is a contradiction to the fact that tangent flows are always backwardly selfsimilar.
 \end{proof}

\begin{prop}\label{prop3}
If the ancient ovals do not occur as limit flow, then all limit flows are selfsimilar.
\end{prop}

\begin{proof}
By White's regularity and structure theory for mean-convex mean curvature flow \cite{White_size,White_nature} (see also Haslhofer-Kleiner \cite{HK_meanconvex}) all limit flows are noncollapsed, convex, and smooth until they become extinct. Thus, by the recent classification of Brendle-Choi \cite{BC} and Angenent-Daskalopoulos-Sesum \cite{ADS} every limit flow must be one of the following: a static plane, a round shrinking sphere, a round shrinking cylinder, a translating bowl solition, or an ancient oval. All except the last one are selfsimilar. This implies the assertion.
\end{proof}

We can now establish the claimed equivalences.

\begin{proof}[{Proof of Theorem \ref{thm_equivalence}}]
If there are only finitely many spherical singularities, then by Proposition \ref{prop2} the ancient ovals cannot occur as limit flow. Together with Proposition \ref{prop3} this implies that all limit flows are selfsimilar. Assume now there are spherical singularities at infinitely many points $X_i$. After passing to a subsequence, the points $X_i$ converge to some point $X$. By the semi-continuity of Huisken's density \cite{Huisken_monotonicity}, the point $X$ must be a singular point. Since for the flow of mean convex surfaces all tangent flows are either round shrinking spheres or round shrinking cylinders by \cite{White_size,White_nature} (see also \cite{HK_meanconvex}), and since spherical singularities are isolated, it follows that $\mathcal M$ has a cylindrical singularity at $X$. Hence, applying Proposition \ref{prop1} we conclude that we get a limit flow that is not selfsimilar.
\end{proof}

To deal with the general case of neck singularities, we employ the recent resolution of the mean-convex neighborhood conjecture for neck singularities from  \cite{CHHW}, which we now quote for convenience of the reader. 
\begin{theorem}[{\cite[Thm. 1.17]{CHHW}}]\label{mean_convex_thm}
Assume $X_0=(x_0,t_0)$ is a space-time point at which the evolution of a smooth closed embedded hypersurface $M\subset\mathbb{R}^{n+1}$ by mean curvature flow has an inward neck singularity. Then there exists an $\varepsilon=\varepsilon(X_0)>0$ such that
\begin{equation}
\quad K_{t_2}\cap B(x_0,\varepsilon)\subseteq K_{t_1}\setminus M_{t_1} 
\end{equation} 
for all $t_0-\varepsilon< t_1< t_2 < t_0+\varepsilon$. Similarly, if the evolution has an outward neck singularity at $X_0$, then there exists some $\varepsilon=\varepsilon(X_0)>0$ such that
\begin{equation}
\quad K_{t_2}'\cap B(x_0,\varepsilon)\subseteq K_{t_1}'\setminus M_{t_1}'
\end{equation}
for all $t_0-\varepsilon< t_1< t_2 < t_1+\varepsilon$. Furthermore, in both cases, any nontrivial limit flow at $X_0$ is either a round shrinking sphere, a round shrinking cylinder, a translating bowl soliton or an ancient oval.
\end{theorem}

\begin{proof}[{Proof of Theorem \ref{thm_equivalence_general}}]
Applying \cite[Thm. B3]{HW}, we get a unit-regular integral Brakke flow $\mathcal{M}$ in a neighborhood of $X_0$, whose support is equal to the outer respectively inner flow, with multiplicity one.

Arguing as in Proposition \ref{prop1}, if there are spherical singularities at $X_i$ with $X_i\rightarrow X_0$, then blowing up by $S(X_i)$ and passing to a limit we would obtain a limit flow $\mathcal{M}^\infty$ at $X_0$, that is not selfsimilar. By the classification of potential limit flows from Theorem \ref{mean_convex_thm} (which generalizes the classification results from \cite{ADS,BC} that have been used in the proof of Proposition \ref{prop3}), $\mathcal{M}^\infty$ has to be an ancient oval.

Conversely, as the argument in Proposition \ref{prop2} only uses the appearance of the ancient ovals as a limit flow, it applies in this context as well. Thus, if the ancient ovals occur as limit flow at $X_0$, there must be a sequence of spherical singularities converging to $X_0$.  
\end{proof}

\bigskip

\bibliographystyle{amsplain}

\vspace{10mm}

{\sc Beomjun Choi, Department of Mathematics, University of Toronto,  40 St George Street, Toronto, ON M5S 2E4, Canada}\\

{\sc Robert Haslhofer, Department of Mathematics, University of Toronto,  40 St George Street, Toronto, ON M5S 2E4, Canada}\\

{\sc Or Hershkovits, Department of Mathematics, Stanford University, 450 Serra Mall, Stanford, CA 94305, USA}\\

\end{document}